\documentclass{amsart}

\usepackage{amssymb,amsmath,amsthm,latexsym}
\usepackage{epsfig}
\usepackage{psfrag}

\newtheorem{Theorem}{\bf Theorem}
\newtheorem{lemma}[Theorem]{\bf Lemma}
\newtheorem{proposition}[Theorem]{\bf Proposition}
\newtheorem{Corollary}[Theorem]{\bf Corollary}

\newtheorem{remark}[Theorem]{\bf Remark}
\newtheorem{theorem}[Theorem]{\bf Theorem}

\def\qed{\hfill$\Box$}
\newcommand{\be}{\begin{equation}}
\newcommand{\ee}{\end{equation}}

\def\qed{\hfill$\Box$}

\def\hpic #1 #2 {\mbox{$\begin{array}[c]{l} 
\epsfig{file=#1,height=#2}\end{array}$}}
\def\wpic #1 #2 {\mbox{$\begin{array}[c]{l} 
\epsfig{file=#1,width=#2}\end{array}$}}

\begin{document}

\title[Skein theory for subfactor planar algebras]{Universal skein theory for finite depth subfactor planar algebras}

\author{Vijay Kodiyalam}
\address{The Institute of Mathematical Sciences, Chennai, India}
\email{vijay@imsc.res.in,tsrikanth@imsc.res.in}
\author{Srikanth Tupurani}



\begin{abstract} We describe an explicit finite presentation for
a finite depth subfactor planar algebra. We also show that such planar
algebras are singly generated with the generator subject to finitely
many relations.
\end{abstract}

\maketitle

\section{Introduction}
The main result of this paper expresses
a subfactor planar algebra
of finite depth as a quotient of a universal planar algebra on finitely many
generators by a planar ideal generated by finitely many relations.
Such a presentation is often referred to as a skein theory for the
planar algebra. In addition, we also show that such a planar algebra
is generated by a single element subject to finitely
many relations.

Our presentation is universal in the following sense.
We specify a small set of `templates' for relations in any finite depth subfactor planar algebra. If $P$ is one such with depth at most $k$, taking a basis of $P_k$ to be a generating set and
specialising these templates to $P$ presents it. 

Skein theories for planar algebras have been the subject of several
studies beginning with \cite{Lnd2002} for the group subfactor planar algebra
and \cite{KdyLndSnd2003} and \cite{KdySnd2006} for irreducible
depth two planar algebras to the more recent
\cite{MrrPtrSny2008}
for the $D_{2n}$ planar algebras, \cite{Bgl2009} for a unified treatment of the ADE planar algebras, \cite{Ptr2009} for the Haagerup subfactor planar algebra and \cite{BglMrrPtrSny2009} for the
extended Haagerup subfactor planar algebra. One of the main results of each of these
papers is a nice skein theory for a finite depth subfactor planar algebra or a family of such.

The methods of this paper do not by any means give any such nice skein theories for finite depth subfactor planar algebras. The point is
to show that all such planar algebras have a skein theory, or equivalently, a finite presentation. In particular, we make no
attempt at being parsimonious with the relations. 

In Section 2 we quickly recall basic definitions and properties of
subfactor planar algebras. Section 3, which makes no mention of planar algebras, is about certain relationships between tangles that
we call templates and certain relationships between templates that
we call consequences. Section 4 gives a finite presentation of a finite
depth subfactor planar algebra. In Section 5 we make a couple of simple observations including the single generation of finite depth subfactor planar algebras.

\section{Subfactor planar algebras}

The purpose of this section is to fix our notations and conventions
regarding planar algebras. We assume that the reader is familiar
with planar algebras as in \cite{Jns1999} or in \cite{KdySnd2004} so we will
be very brief.

Planar algebras are collections of vector spaces equipped with an action by the
`coloured operad of planar tangles'.
The vector spaces are indexed by
the set $Col = \{0_+, 0_-, 1, 2, \cdots \}$, whose elements  are called colours.
We endow this set with the partial order that restricts to the usual order on ${\mathbb N}$ and such that $0_\pm$ are incomparable
and less than 1.

We will not define a tangle but merely note the following features.
Each tangle has an external box, denoted $D_0$, 
and a (possibly empty) ordered collection of 
internal non-nested boxes denoted
$D_1$, $D_2$, $\cdots$.
Each box has an even number (again possibly 
0) of points marked on its boundary. A box with $2n$ points on its boundary is called an $n$-box or said to be of colour $n$.
There is also given a collection of disjoint curves each of which is either 
closed,
or joins a marked point on one of the boxes to another such.
For each box having at least one marked point on its boundary, one
of the regions ( = connected components of the complement of the boxes and
curves) that impinge on its boundary is distinguished and marked
with a $*$ placed near its boundary.
The whole picture is to be planar and each marked point on a box must be
the end-point of one of the curves.
Finally, there is given a chequerboard shading of the regions such 
that 
the $*$-region of any box is shaded white.
A $0$-box is said to be  $0_+$ box if the region touching its boundary is 
white and a  $0_-$ box
otherwise.
A $0$ without the $\pm$ qualification will always refer to $0_+$.
A tangle is said to be an $n$-tangle if its external box is of colour $n$.
Tangles are defined only upto
a planar isotopy preserving the $*$'s, the shading and the ordering of the 
internal boxes.

We illustrate several important tangles in Figure \ref{fig:imptangles}.
\begin{figure}[!htb]
\psfrag{n}{\tiny $n$}
\psfrag{m-n+p}{\tiny $m-n+p$}
\psfrag{n-m+p}{\tiny $n-m+p$}
\psfrag{m+n-p}{\tiny $m+n-p$}
\psfrag{i}{\tiny $i$}
\psfrag{2n-i}{\tiny $2n-i$}
\psfrag{j}{\tiny $j$}
\psfrag{$ER_{n+i}^n$ : Right expectations}{$ER_{n+j}^n$ : Right expectation}
\psfrag{$ER$}{$ER_{1}^{0_-}$ : Right expectation}
\psfrag{2n}{\tiny $2n$}
\psfrag{2}{\tiny $2$}
\psfrag{D_1}{\small $D_2$}
\psfrag{D_2}{\small $D_1$}
\psfrag{$I_n^{n+1}$ : Inclusion}{$I_n^{n+j}$ : Inclusion}
\psfrag{$I_{0_-}^1$ : Inclusion}{$I_{0_-}^1$ : Inclusion}
\psfrag{$EL(i)_{n+i}^{n+i}$ : Left expectations}{$EL_{n}^{n}$ : Left expectation}
\psfrag{$M_{n,n}^n$ : Multiplication}{$M_{m,n}^{p}$ : Multiplication}
\psfrag{$TR_n^0$ : Trace}{$TR_n^0$ : Trace}
\psfrag{$R^{n+1}_{n+1}$ : Rotation}{$R_{n+1}^{n+1}$ : Rotation}
\psfrag{$1^n$ : Multiplicative identity}{$1^n$ : Mult. identity}
\psfrag{$ER_{n+1}^{n+1}$ : Right expectation}{$ER_{n+1}^{n+1}$ : Right expectation}
\psfrag{$I_n^n$ : Identity}{$I_n^n$ : Identity}
\psfrag{$E^{n+2}$ : Jones projections}{$E^{n+2}$ : Jones proj.}
\includegraphics[height=8.5cm]{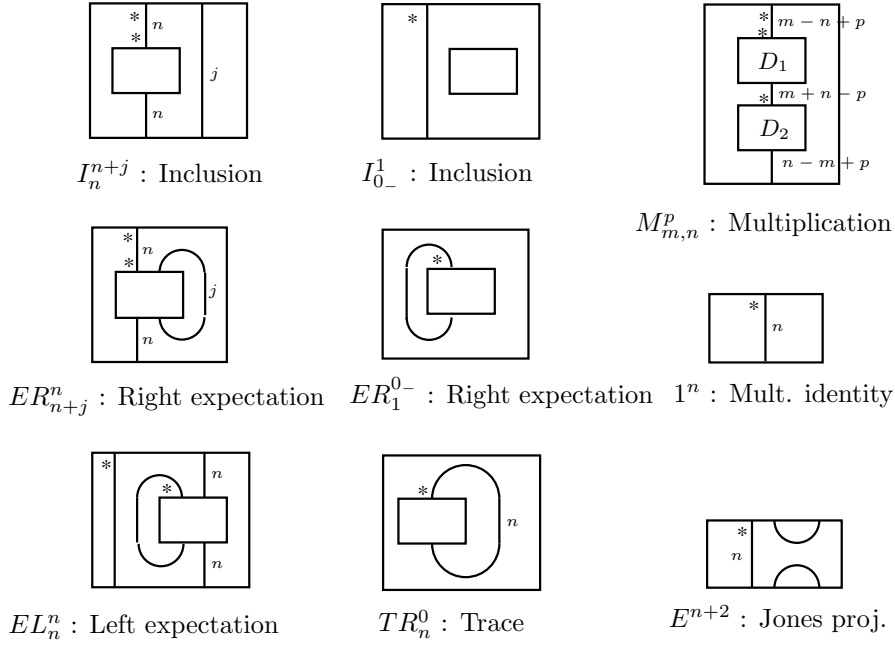}
\caption{Some important tangles ($m,n,j \geq 0, |m-n| \leq p \leq m+n$)}
\label{fig:imptangles}
\end{figure}
This figure,  uses the following notational device introduced in \cite{KdySnd2009}.
A strand in a tangle with a {\em non-negative integer}, say $t$, adjacent to it will indicate
a $t$-cable of that strand, i.e., a parallel cable of $t$ strands, in
place of the one actually drawn. 

A useful labelling convention for tangles 
is to decorate its tangle symbol,
such as $I,EL,M$ or $TR$, with subscripts and a superscript that give the colours
of its internal boxes and external box respectively. With this, we may dispense
with showing the shading, which is then unambiguously determined.

The basic operation that one can perform on tangles is substitution of one
into a box of another. If $T$ is a tangle that has some internal  boxes
$D_{i_1}, \cdots, D_{i_j}$ of colours $n_{i_1}, \cdots, n_{i_j}$ and if
$S_1, \cdots, S_j$ are arbitrary tangles of colours $n_{i_1}, \cdots, 
n_{i_j}$, then we may substitute $S_t$ into the box $D_{i_t}$ of $T$ for 
each $t$ - such that the `$*$'s match' - to get a new tangle that will be 
denoted $T \circ_{(D_{i_1}, \cdots,D_{i_j})}
(S_1, \cdots, S_j)$.
The collection of tangles along with the substitution operation is called
the coloured operad of planar tangles.

A planar algebra $P$ is an algebra over the coloured operad of planar tangles.
By this, is meant the following: $P$ is a collection $\{P_n\}_{n \in Col}$ of vector spaces
and linear maps $Z_T^P : P_{n_1} \otimes P_{n_2} \otimes \cdots \otimes 
P_{n_b} 
\rightarrow P_{n_0}$ for each $n_0$-tangle $T$ with internal boxes of colours
$n_1,n_2, \cdots,n_b$. The collection of maps is to be `compatible with 
substitution of tangles and renumbering of internal boxes' in an obvious 
manner.
For a planar algebra $P$, each $P_n$ acquires the structure of an
associative, unital algebra with multiplication defined using the
tangle $M_{n,n}^n$ and unit defined to be $1_n = Z^P_{1^n}(1)$.

Among planar algebras, the ones that we will be interested in are the
subfactor planar algebras.
These are complex, finite-dimensional and connected in the sense that each $P_n$ is a 
finite-dimensional complex vector space and $P_{0_\pm}$ are one dimensional.
They have a positive modulus $\delta$, meaning that closed loops in a 
tangle $T$ contribute
a multiplicative factor of $\delta$ in $Z_T^P$.
They are spherical in that for a $0$-tangle $T$, the function $Z_T^P$ is not 
just
planar isotopy invariant but also an isotopy invariant of the tangle regarded
as embedded on the surface of the two sphere.
Further, each $P_n$ is a $C^*$-algebra in such a way that for an $n_0$-tangle 
$T$ with internal boxes of colours
$n_1,n_2, \cdots,n_b$ and for $x_i \in P_{n_i}$, the equality
$Z_T^P(x_1 \otimes \cdots \otimes x_b)^* ~=~
Z_{T^*}^P(x_1^* \otimes \cdots \otimes x_b^*)$ holds,
where $T^*$ is the adjoint of the tangle $T$ - which, by definition, is 
obtained from $T$ by reflecting it.
Finally, the trace $\tau : P_n \rightarrow {\mathbb C} = P_{0}$ 
defined by:
\begin{eqnarray*}
\tau(x) ~=~  \delta^{-n} Z^P_{TR_n^{0}}(x)
\end{eqnarray*}
is postulated to be a faithful, positive (normalised) trace for each $n \geq 0$.

Any subfactor planar algebra $P$ (of modulus $\delta$) contains the distinguished Jones projections
$e_n \in P_n$ for $n \geq 2$ defined by $e_n = \delta^{-1}Z_{E^n}^P(1)$
and their non-normalised versions $E_n = Z_{E^n}^P(1)$.
A subfactor planar algebra $P$ is said to be of finite depth if there is
a positive integer $k$ such that $P_{k+1} = P_kE_{k+1}P_k$ and the
smallest such $k$ is said to be the depth of $P$. 


The following proposition is well-known. We only give a proof for
completeness and since it is completely planar-algebraic. Note the absence of any assumptions on the planar algebra.

\begin{proposition}\label{tensor}
Let $P$ be any planar algebra and suppose that for
some positive integer $k$, $1_{k+1} \in P_kE_{k+1}P_k$. For all $m,n \geq k$, there is an isomorphism of $P_{k-1}-P_{k-1}$-bimodules, $$P_m \otimes_{P_{k-1}} P_n \cong P_{m+n-(k-1)}.$$
\end{proposition}

\begin{proof}
Consider the 
tangles $T^n$ defined for $n \in Col$ as in Figure
\ref{tl}.
In this and all subsequent tangle figures, we suppress drawing the external
box of tangles and adopt the convention that the $*$ of the
external box (if it is a $k$-box with $k >0$) is at the top left corner.
Shaded regions of a tangle will be to the left traversing any string
along the direction indicated on it.
\begin{figure}[!h]
\begin{center}
\psfrag{1}{\large $D_1$}
\psfrag{2}{\large $D_2$}
\psfrag{l}{\large $n$}
\psfrag{lmkp1}{\large $D_{n-k+1}$}
\psfrag{k-1}{\large $k-1$}
\psfrag{k-l}{\large $k-n$}
\psfrag{v}{\huge $\vdots$}
\psfrag{=}{\huge $=$}
\resizebox{5.5cm}{!}{\includegraphics{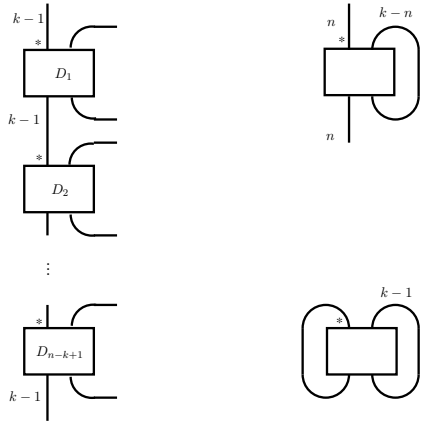}}
\end{center}
\caption{The tangles $T^n$ for $n \geq k$, $0 \leq n < k$ and $n=0_-$}
\label{tl}
\end{figure}
The tangle $T^n$ is an $n$-tangle with $n-k+1$ internal boxes for $n \geq k$ and
1 internal box for $n < k$, all of colour $k$. Note that for
$n < k$, $T^n = ER_k^n$ while $T^k = I^k_k$.

From $1_{k+1} \in P_kE_{k+1}P_k$ we see easily -  see the technique of proof of Lemma 5.7
of \cite{KdyLndSnd2003} - that for all $n \geq k$,
$P_{n+1} = P_kE_{k+1}E_{k+2}\cdots E_{n+1}P_n$ and then by
induction that $P_{n+1} = P_kE_{k+1}E_{k+2}\cdots E_{n+1}P_kE_{k+1}E_{k+2}\cdots E_{n}P_k \cdots \cdots  P_kE_{k+1}P_k$. Expressed pictorially, this
yields the surjectivity of $Z_{T^n}^P$ for all $n \geq k$.

Now consider the 
tangle $M = M_{m,n}^{m+n-(k-1)}$.
Thus $Z_M^P : P_m \otimes P_n \rightarrow P_{m+n-(k-1)}$ and
a little thought shows that this is a $P_{k-1}-P_{k-1}$-bimodule map that factors through $P_m \otimes_{P_{k-1}} P_n$. Surjectivity of this map follows from the tangle equation  $M\circ_{(D_1,D_2)}(T^m,T^n) = T^{m+n-(k-1)}$.

The proof of injectivity uses the tangles $W = W^n_{n,2n-k+1}$ and
$W^*$ of Figure \ref{tanglew}.
\begin{figure}[!h]
\begin{center}
\psfrag{D1}{\large $D_1$}
\psfrag{D2}{\large $D_2$}
\psfrag{l}{\large $n$}
\psfrag{k-1}{\large $k-1$}
\psfrag{n-k+1}{\large $n-k+1$}
\psfrag{n}{\large $n$}
\resizebox{6cm}{!}{\includegraphics{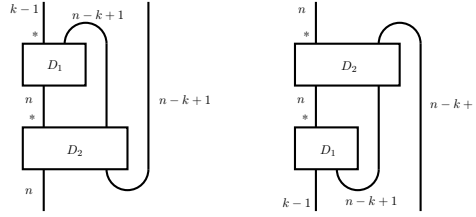}}
\end{center}
\caption{The tangles $W$ and $W^*$}
\label{tanglew}
\end{figure}
First use the surjectivity above for $m=n$ to conclude
that there exist $x_i,y_i \in P_n$, for $i \in I$ - a finite set - such that 
$1_{2n-(k-1)} = \sum_{i \in I} Z_M^P(x_i \otimes y_i)$. 
Hence, for any $v \in P_n$, $Z^P_W(v,1_{2n-(k-1)}) = \sum_{i \in I} Z_{W \circ_{D_2} M}^P(v \otimes x_i \otimes y_i)$. Equivalently, for all
%
%
%
%
$v \in P_n$, we have $v = \sum_{i \in I} Z_{ER_n^{k-1}}(vx_i)y_i$. 

Now, we claim that
if $\sum_{j \in J} u_j \otimes v_j \in ker(Z_M^P)$, then,
$$
\sum_{j \in J} u_j \otimes v_j = \left( \sum_{i \in I,j \in J} u_j \otimes Z_{ER_n^{k-1}}(v_jx_i)y_i\right) - \left( \sum_{i \in I,j \in J} u_j  Z_{ER_n^{k-1}}  (v_jx_i) \otimes y_i\right).
$$
In fact, the left hand side equals the first term on the right hand side 
while the second term on the right vanishes since for each $i \in I$, the sum $ \sum_{j \in J} u_j  Z_{ER_n^{k-1}}  (v_jx_i)$ is of the form $Z^P_{W^* \circ_{D_2} M}(x_i \otimes \sum_{j \in J} u_j \otimes v_j)
= Z^P_{W^*}(x_i \otimes Z^P_M(\sum_{j \in J} u_j \otimes v_j)) = 0.
$ 

The displayed equation expresses $
\sum_{j \in J} u_j \otimes v_j$ as an element in the kernel of the natural map $P_m \otimes P_n \rightarrow P_m \otimes_{P_{k-1}} P_n$
and concludes the proof.
\end{proof}
 
We will need the following corollary whose proof follows easily by
induction using Proposition \ref{tensor}. 
\begin{Corollary}\label{cor:tensor}
Let $P$ be any planar algebra and suppose that for
some positive integer $k$, $1_{k+1} \in P_kE_{k+1}P_k$.
Then, for all $n \geq k$ there is a $P_{k-1}-P_{k-1}$-bimodule
isomorphism
$$
P_k \otimes_{P_{k-1}} P_k \otimes_{P_{k-1}} \cdots \otimes_{P_{k-1}}
P_k \cong P_n,
$$
where there are $n+1-k$ $P_k$'s on the left.\qed
\end{Corollary}

\section{Templates and consequences}

A template is an ordered pair of tangles $(S,T)$ of the same colour but will be written as  a tangle implication $S \Rightarrow T$.
Given any set of templates, we will be interested in their `consequences' which are by definition those that can be obtained from 
them using (i) `reflexivity' (ii) `transitivity' and (iii) `composition on the outside', i.e., elements
of the smallest set of templates containing the original set and such that (i) all $T\Rightarrow T$ are in the set, (ii) if $S \Rightarrow T$ and $T \Rightarrow V$ are in the set, so is
$S \Rightarrow V$, and (iii) if $W$ is an arbitrary $(n_0;n_1,\cdots,n_b)$ tangle and $S_i \Rightarrow T_i$ are in the set with colour $n_i$, then, $W \circ_{(D_1,\cdots,D_b)}(S_1,\cdots,S_b) \Rightarrow W \circ_{(D_1,\cdots,D_b)}(T_1,\cdots,T_b)$ is also in the set.
%

%


%



For this paper we need a particular collection of templates
shown in Figure \ref{templates} which we will refer to as the basic templates. Here $k$ is a fixed positive integer.
\begin{figure}[!h]
\begin{center}
\psfrag{M}{\Huge Modulus: $C^{0_\pm} \Leftrightarrow 1^{0_\pm}$}
\psfrag{I}{\Huge Identity: }
\psfrag{IN}{\Huge Inclusion: }
\psfrag{J}{\Huge Jones proj.: $I_{n}^k \circ E^{n} \Rightarrow I_k^k$}
\psfrag{MU}{\Huge Multiplication: $M_{k,k}^k \Rightarrow I_k^k$}
\psfrag{CE}{\Huge Cond. exp.: $I_{k-1}^k \circ E_{k}^{k-1} \Rightarrow I_k^k$}
\psfrag{DE}{\Huge Depth: $1^{k+1} \Rightarrow T^{k+1}$}
\psfrag{SH}{\Huge Shift: $SH_{k}^{k+2} \Rightarrow T^{k+2}$}
\psfrag{=}{\huge $\Rightarrow$}
\psfrag{eq}{\huge $\Leftrightarrow$}
\psfrag{k}{\large $k$}
\psfrag{k+1}{\large $k+1$}
\psfrag{l}{\large $n$}
\psfrag{l-1}{\large $n-2$}
\psfrag{k-1}{\large $k-1$}
\psfrag{k-l-1}{\large $k-n$}
\psfrag{del}{\huge $1^{0_-}$}
\psfrag{del-}{\huge $1^{0_+}$}
\resizebox{12.5cm}{!}{\includegraphics{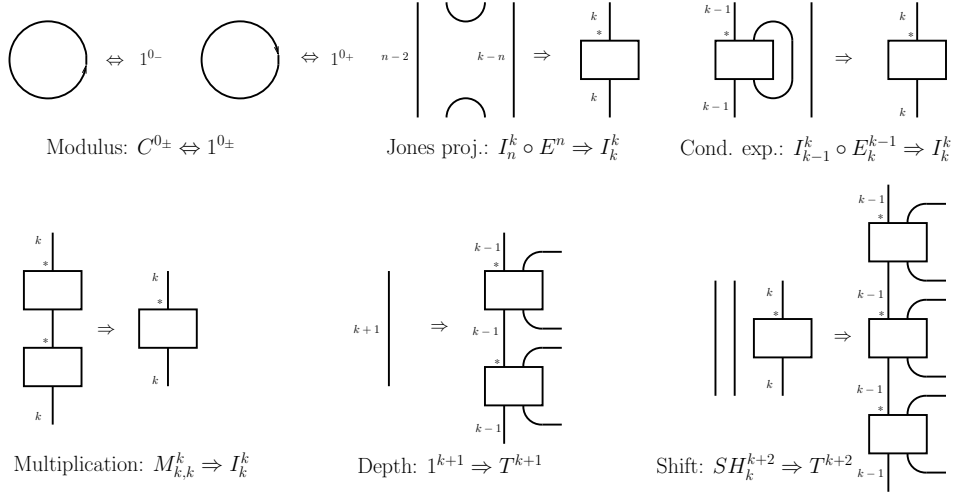}}
\end{center}
\caption{The basic templates ($2 \leq n \leq k$ for the Jones projections)}
\label{templates}
\end{figure}
Note that Figure \ref{templates} names each of the templates, shows them as tangle implications, and in the process, defines some tangles.

We begin with a simple but very useful lemma which we will refer to
later as `removing loops'.

\begin{lemma}\label{remove}
Let $S \Rightarrow T$ be any template such that the tangle $S$ has a contractible loop somewhere in it and let $\tilde{S}$ be $S$ with
the loop removed. The modulus templates together with
$S \Rightarrow T$ have as consequence $\tilde{S} \Rightarrow T$.
\end{lemma}

\begin{proof}
Suppose that the contractible loop of $S$ lies in a white region.
Let $W$ be the tangle obtained from $S$ by replacing the contractible
loop with a $0_+$ box numbered $b+1$, where $S$ has $b$ internal
boxes. Then it is clear that $S = W \circ_{D_{b+1}}(C^{0_+})$ while
$\tilde{S} = W \circ_{D_{b+1}}(1^{0_+})$. Since the modulus tangle
gives $1^{0_+} \Rightarrow C^{0_+}$, by composing on the outside
with $W$, we get $\tilde{S} \Rightarrow S$ and so by transitivity
$\tilde{S} \Rightarrow T$. A similar proof applies when the loop lies
in a black region.
\end{proof}

The main result of this section is an omnibus theorem listing various consequences of the templates of Figure \ref{templates}.
While all the consequences are written as tangle implications, we emphasise
that the proofs are purely pictorial.
Recall the tangles $T^n$ defined for $n \in Col$  in Figure
\ref{tl}.

\begin{theorem}\label{consequences}
The following templates are all consequences of the basic templates of Figure \ref{templates}.
\begin{enumerate}
\item $1^k \Rightarrow T^k$.
\item $I_k^{k+1} \Rightarrow T^{k+1}$.
\item For all $n \in Col$, $ER^n_{n+1} \circ T^{n+1} \Rightarrow T^{n}$.
\item For any $n \geq k$, $I_n^{n+1} \circ T^n \Rightarrow T^{n+1}$.
\item For any $n \geq k$, $I_k^n \Rightarrow T^n$ and $1^n \Rightarrow T^n$.
\item For any $n \geq k$, $M^n_{n,n} \circ_{(D_1,D_2)}(T^{n}, T^n)\Rightarrow T^{n}$.
\item $1^{0_\pm} \Rightarrow T^{0_\pm}$ and for any $n \geq 2$, $E^n \Rightarrow T^n$.
\item For any $n \geq k$ and any Temperley-Lieb tangle $Q^n$, $Q^n \Rightarrow T^n$.
\item For any $n \geq k$, $SH_n^{n+2} \circ T^n \Rightarrow T^{n+2}$.
\item For any $n \geq 1$, $EL_n^n \circ T^n \Rightarrow T^{n}$.
\item For all $n \in Col$, $I_n^{n+1} \circ T^n \Rightarrow T^{n+1}$.
\item For all $n \in Col$, $M^n_{n,n} \circ_{(D_1,D_2)}(T^{n}, T^n)\Rightarrow T^{n}$.
\end{enumerate}
\end{theorem}

\begin{proof} (1) According to the depth template  $1^{k+1} \Rightarrow
T^{k+1}$. Applying $ER_{k+1}^k$ on both sides yields
$ER_{k+1}^k \circ 1^{k+1} \Rightarrow ER_{k+1}^k \circ T^{k+1} = M_{k,k}^k$.
Since $ER_{k+1}^k \circ 1^{k+1}$ is $1^k$ with a contractible  loop on the right,
we may remove this loop by Lemma~\ref{remove} and conclude that
$1^k \Rightarrow I_k^k$.\\
(2) Since $1^{k+1} \Rightarrow T^{k+1}$ and $I_k^{k+1}
\Rightarrow I_k^{k+1}$ we may apply the multiplication tangle
$M_{k+1,k+1}^{k+1}$ to the outside to get
$$M_{k+1,k+1}^{k+1} \circ_{(D_1,D_2)} (1^{k+1},I_k^{k+1}) \Rightarrow
M_{k+1,k+1}^{k+1} \circ_{(D_1,D_2)} (T^{k+1},I_k^{k+1}).$$
This may also be written as $I_k^{k+1} \Rightarrow T^{k+1} \circ_{D_2} M^k_{k,k}$. Since $M^k_{k,k} \Rightarrow I^k_k$, 
we have $T^{k+1} \circ_{D_2} M^k_{k,k} \Rightarrow T^{k+1} \circ_{D_2} I^k_k = T^{k+1}.$ Now appeal to transitivity.\\
(3) Suppose that $n < k$. Then $ER^n_{n+1} \circ T^{n+1} = T^n$, so the
asserted result is clear by reflexivity. If $n \geq k$, there are two cases
depending on the parity of $n-k$. 
These cases are shown on the left in Figure \ref{ericircn}.
%
%
\begin{figure}[!h]
\begin{center}
\psfrag{1}{\large $D_1$}
\psfrag{2}{\large $D_2$}
\psfrag{l}{\large $n$}
\psfrag{lmkp1}{\large $D_{n-k+2}$}
\psfrag{lmk}{\large $D_{n-k+1}$}
\psfrag{tp1}{\large $D_{t+1}$}
\psfrag{tp2}{\large $D_{t+2}$}
\psfrag{t}{\large $D_t$}
\psfrag{2t}{\huge $n-k=2t$}
\psfrag{2t-1}{\huge $n-k=2t-1$}
\psfrag{2t+1}{\huge $n-k=2t+1$}
\psfrag{k-1}{\large $k-1$}
\psfrag{k-l}{\large $k-n$}
\psfrag{v}{\huge $\vdots$}
\psfrag{=}{\huge $=$}
\resizebox{12cm}{!}{\includegraphics{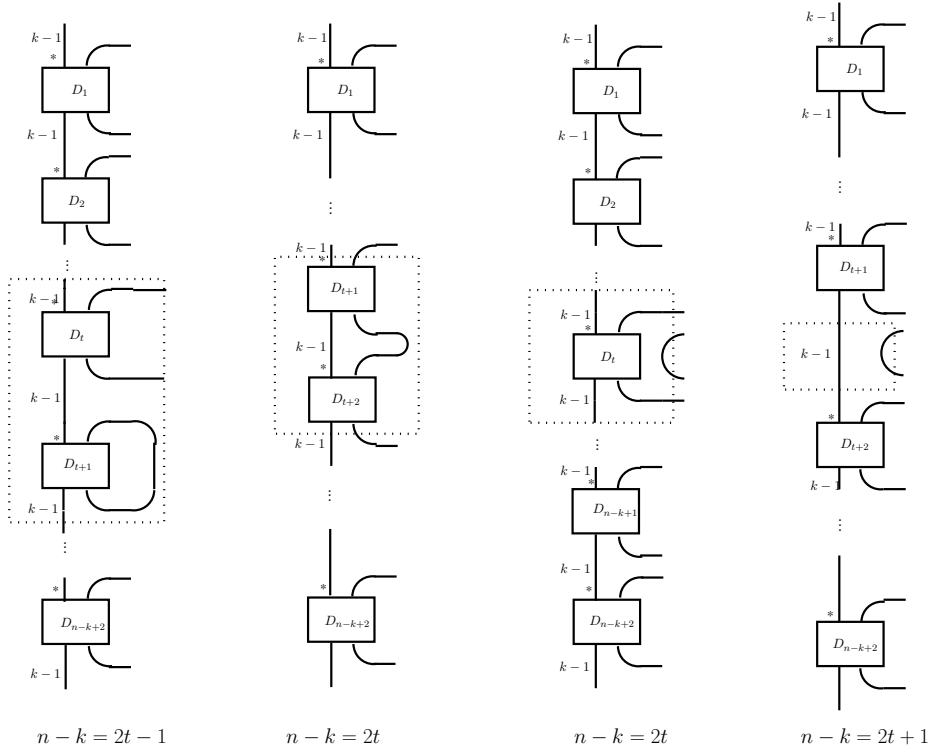}}
\end{center}
\caption{$ER^n_{n+1} \circ T^{n+1}$ and $I^{n+1}_n \circ T^{n}$}
\label{ericircn}
\end{figure}


We see that each is obtained by inserting a $k$-tangle into a box
of $T^n$ and using the multiplication and conditional expectation
templates, this $k$-tangle, in each case, implies $I^k_k$.
\\
(4) Again, there are two cases according to the parity of $n-k$ 
which are shown on the  right in Figure \ref{ericircn}.
If $n-k=2t$, we see that $I^{n+1}_n \circ T^{n} = W \circ I_k^{k+1}$
for a suitable tangle $W$ (where $W$ has a $k+1$-box indicated by the dotted line and the rest of it looking like $T^n$). Note now that the
inclusion template gives $I_k^{k+1} \Rightarrow T^{k+1}$ and therefore $W \circ I_k^{k+1} \Rightarrow W \circ T^{k+1}$. It remains only to note that $W \circ T^{k+1} = T^{n+1}$ and use transitivity to complete the
proof in this case.
The  case $n-k=2t+1$ is even easier. Here $I^{n+1}_n \circ T^{n} = T^{n+1} \circ_{D_{t+2}}1^k$. Since $1^k \Rightarrow I^k_k$, we get
$I^{n+1}_n \circ T^{n} = T^{n+1} \circ_{D_{t+2}}1^k \Rightarrow T^{n+1} \circ_{D_{t+2}}I^k_k = T^{n+1}$.\\
(5)  We have by reflexivity that $I^k_k \Rightarrow T^k$.
Applying (4) inductively shows that for all $n \geq k$, $I_k^n \Rightarrow T^n$. A similar proof beginning with (1) shows that $1^n \Rightarrow T^n$.\\
(6) For $n=k$, this is just the multiplication template. For $n > k$, a
little doodling should convince the reader that $M^n_{n,n} \circ_{(D_1,D_2)}(T^{n}, T^n) = ER^n_{2n-k+1} \circ T^{2n-k+1}$.
Transitivity, (3) and induction finish the proof. \\
(7) Begin with the identity template $1^k \Rightarrow I^k_k$
and apply $ER^{0_\pm}_k$ to both sides to get $ER^{0_\pm}_k \circ 1^k \Rightarrow ER^{0_\pm}_k \circ I^k_k = ER^{0_\pm}_k = T^{0_\pm}$. The left side of this implication is a $0^\pm$-tangle which is a collection of loops which may be removed by Lemma \ref{remove} to yield $1^{0_\pm} \Rightarrow T^{0_\pm}$. 
A very similar proof beginning with the Jones projection templates
gives $E^n \Rightarrow T^n$ for $2 \leq n \leq k$.
To show that $E^n \Rightarrow T^n$ for $n > k$, consider the
following chain of implications. 
\begin{eqnarray*}
 E^n &=& ER_{2n-k-1}^n \circ M_{n-1,n-1}^{2n-k-1} \circ_{(D_1,D_2)}(1^{n-1},1^{n-1})\\
&\Rightarrow& ER_{2n-k-1}^n \circ M_{n-1,n-1}^{2n-k-1} \circ_{(D_1,D_2)}(T^{n-1},T^{n-1})\\
&=& ER_{2n-k-1}^n \circ T^{2n-k-1}\\
&\Rightarrow& T^n,
\end{eqnarray*}
where the first implication is a consequence of (5) and the second, of (3) and induction.\\
(8) This is an easy corollary of (4), (6) and (7).\\
(9) Induce on $n$, with the basis case being asserted by the
shift template. For $n > k$, 
\begin{eqnarray*}
 SH_n^{n+2} \circ T^n &=& M^{n+2}_{n+1,k+2} \circ_{(D_1,D_2)} (SH_{n-1}^{n+1} \circ T^{n-1},SH_k^{k+2})\\
&\Rightarrow& M^{n+2}_{n+1,k+2} \circ_{(D_1,D_2)} (T^{n+1},T^{k+2})\\
&\Rightarrow& T^{n+2},
\end{eqnarray*}
where the last implication uses the multiplication and conditional
expectation templates together with a suitable outside composition.\\
(10) First suppose that $n \geq k$.
Begin with the conclusion $SH_n^{n+2} \circ T^n \Rightarrow T^{n+2}$ in (9). Let $Q^{n+2}$ and $Q^{*n+2}$ be the Temperley-Lieb tangles  shown in Figure \ref{tangleq},
\begin{figure}[!h]
\begin{center}
\psfrag{n}{\large $n$}
\psfrag{n-1}{\large $n-1$}
\psfrag{k-n}{\large $k-n$}
\resizebox{10cm}{!}{\includegraphics{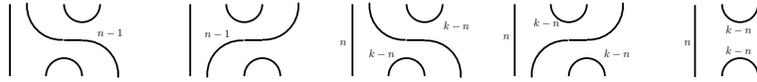}}
\end{center}
\caption{The tangles $Q^{n+2}$, $Q^{*n+2}$, $K^{2k-n+1}$, $K^{*2k-n+1}$and $L^{2k-n}$}
\label{tangleq}
\end{figure}
so that, by (8), $Q^{n+2} \Rightarrow T^{n+2}$ and $Q^{*n+2} \Rightarrow T^{n+2}$ . Then, 
%
%
%
%
%
with $M = M^{n+2}_{n+2,n+2,n+2}$ denoting the iterated 
multiplication tangle we have,
$$M \circ (Q^{n+2},SH_n^{n+2} \circ T^n,Q^{*n+2})
\Rightarrow M \circ(T^{n+2},T^{n+2},T^{n+2}) \Rightarrow T^{n+2}.$$ 
(For typographical convenience, we have omitted the subscripts to
$\circ$). Hence $
ER_{n+2}^n \circ M \circ (Q^{n+2},SH_n^{n+2},Q^{*n+2})
 \Rightarrow 
 ER_{n+2}^n \circ T^{n+2}
 \Rightarrow
  T^n.
$
The left hand side of this chain of implications is $EL^n_n \circ T^n$ with a loop at its right; therefore, using Lemma \ref{remove}, we get the desired result.
For
$1 \leq n < k$, merely apply $ER^n_k$ to both sides of $EL_k^k \circ T^k \Rightarrow T^{k}$. \\
(11) In view of (4), we only need consider the case $n < k$. If $n=0_-$, this is just the case $n=1$ of (10). So suppose that $0 \leq n < k$. Let $t = 2k-n+1$.
Start with $I_k^{t} \Rightarrow T^{t}$ deduced inductively from (4). Let $K^{t}$ and $K^{*t}$ be the Temperley-Lieb tangles in Figure \ref{tangleq} so that by (8), $K^{t} \Rightarrow T^{t}$ and $K^{*t} \Rightarrow T^{t}$.
Now, with $M = M^{t}_{t,t,t}$,
$M \circ
(K^{t},I_k^{t},K^{*t})
\Rightarrow M \circ
(T^{t},T^{t},T^{t}) \Rightarrow T^{t}.
$
Applying $ER_{t}^{n+1}$ to both sides of this and removing the
$k-n$ loops that arise on the left hand side, we get the desired
conclusion using (3).\\
(12) In view of (4), we may assume that $0 \leq n < k$. Let $u = 2k-n$ and $M = M_{u,u,u,u,u}^{u}$.
Then, with $L^{u}$ as in Figure \ref{tangleq}, $$M \circ (L^{u},I^{u}_k,L^{u},I^{u}_k,L^{u}) 
\Rightarrow M \circ (T^{u},T^{u},T^{u},T^{u},T^{u}) \Rightarrow T^{u}.$$ As in (11), applying $ER_{u}^n$ to both
sides and removing the $k-n$ loops gives the desired conclusion.
The case $n=0_-$ seems to be surprisingly complicated.
Begin with $I_k^{2k} \Rightarrow T^{2k}$ deduced inductively using (4). Applying $SH_{2k}^{4k}$ (with the obvious meaning)
on both sides and using (9) repeatedly gives $SH_{2k}^{4k} \circ I_k^{2k} \Rightarrow T^{4k}$. This is shown pictorially
in Figure \ref{shhh}.
\begin{figure}[!h]
\begin{center}
\psfrag{k}{\large $k$}
\psfrag{i}{\huge $\Rightarrow \ \ T^{4k}$}
\resizebox{3cm}{!}{\includegraphics{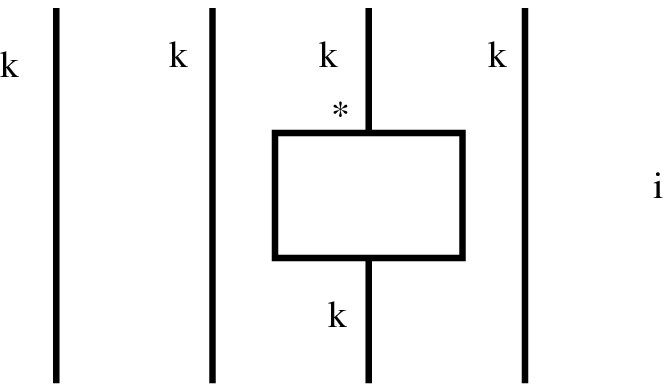}}
\end{center}
\caption{}
\label{shhh}
\end{figure}
Multiplying by appropriate Temperley-Lieb tangles above and below and using (8) and (6), we get the template on the left of Figure \ref{shhh2}. Applying $ER_{4k}^{2k}$  to both sides and removing the $k$
loops that arise and then applying $I_{2k}^{2k+2}$ to both sides gives
the template in the middle in Figure \ref{shhh2}.
\begin{figure}[!h]
\begin{center}
\psfrag{k}{\large $k$}
\psfrag{i}{\huge $\Rightarrow \ \ T^{4k}$}
\psfrag{j}{\huge $\Rightarrow \ \ T^{2k+2}$}
\resizebox{10.5cm}{!}{\includegraphics{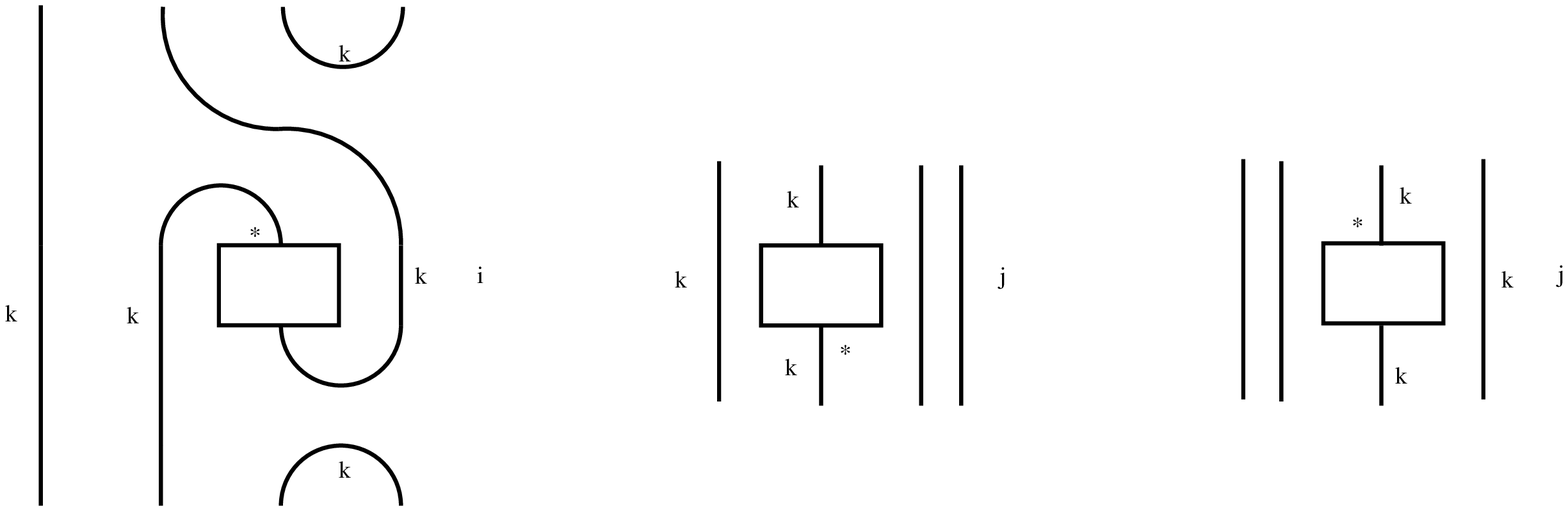}}
\end{center}
\caption{}
\label{shhh2}
\end{figure}
A much easier proof shows that $I^{2k+2}_{k+2} \circ SH^{k+2}_k \Rightarrow T^{2k+2}$ is a consequence of the basic templates which is the right side template in Figure \ref{shhh2}. Now multiplying by
appropriate Temperley-Lieb tangles above, in-between and below, \begin{figure}[!h]
\begin{center}
\psfrag{k-1}{\large $k-1$}
\psfrag{i}{\huge $\Rightarrow \ \ T^{2k+2}$}
\psfrag{j}{\huge $\Rightarrow \ \ T^{2k+2}$}
\resizebox{3cm}{!}{\includegraphics{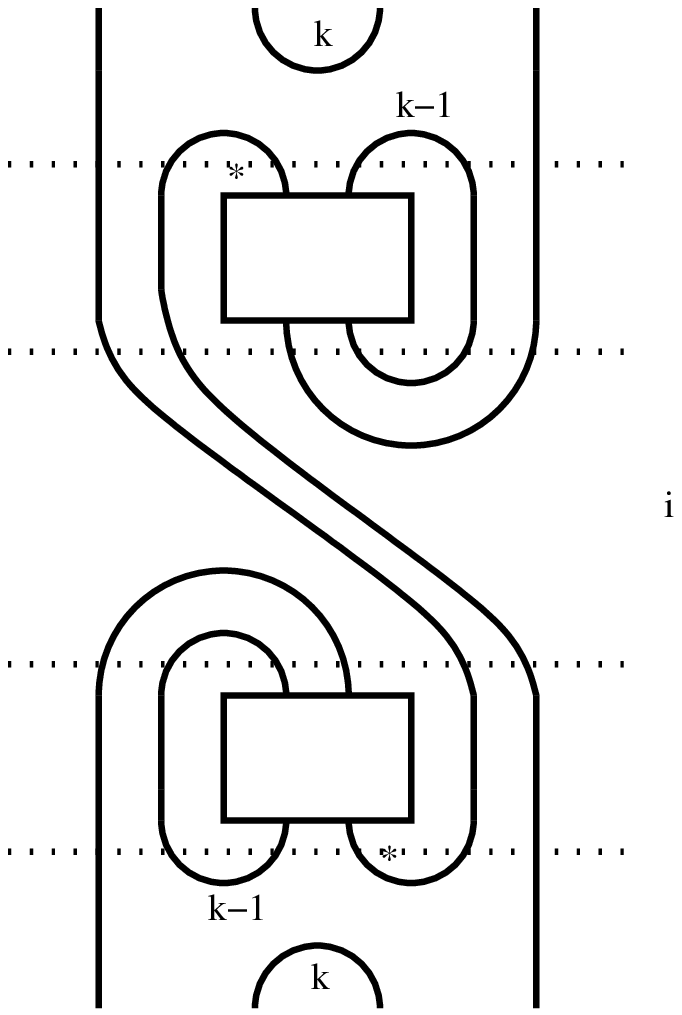}}
\end{center}
\caption{}
\label{shhh3}
\end{figure}
and using (6)
we get the template of  Figure~\ref{shhh3}.
Finally, applying $ER_{2k+2}^{0_-}$ to this template, removing the $k$ loops and using (3) repeatedly yields the desired
result.
\end{proof}

\section{The main theorem}

Let $P$ be a planar algebra and $B \subseteq P$, i.e.,  $B = \coprod_{n \in Col} B_n$ where $B_n \subseteq P_n$ for all $n \in Col$.
Given the pair $(P,B)$,
each $(n_0;n_1,\cdots,n_b)$-tangle $T$ then determines a certain
subspace $R_{(P,B)}(T) \subseteq P_{n_0}$ defined to be (i) the span
of all $Z^P_T(x_1 \otimes \cdots \otimes x_b)$ for $x_i \in B_{n_i}$ if $b>0$ or (ii) the span of 
$Z^P_T(1)$ if $b=0$.
%
%
A template $S \Rightarrow T$ is said to hold for 
the pair $(P,B)$ 
if
$R_{(P,B)}(S) \subseteq R_{(P,B)}(T)$. 
It is easy to see that if a set
of templates holds for $(P,B)$ then so do all their consequences.
%
 Our first observation is fairly
easy to see.

\begin{proposition}
If $P$ is a subfactor planar algebra of finite
depth at most $k$, 
and $B= B_k$ is a basis of $P_k$, then
all the templates of Figure \ref{templates}
hold for $(P,B)$.
\end{proposition}

\begin{proof} The modulus templates hold for $(P,B)$ since $P$ has non-zero modulus.
The Jones projections, multiplication and the
conditional expectation templates hold for $(P,B)$ since
their right sides are all the identity tangle $I^k_k$ and $B$ is a basis of $P_k$. The depth and shift templates hold
since the tangles on their right, namely $T^{k+1}$ and $T^{k+2}$, surject onto their ranges (from the proof of Proposition \ref{tensor}).
\end{proof}


%

%


Before proceeding we briefly recall (see \cite{KdyLndSnd2003} for detailed explanations) what a presentation of a planar
algebra is. Given a label set $L = \coprod_{n \in Col} L_n$, there
is a universal planar algebra on $L$, denoted by $P(L)$. By definition,   for all $n \in Col$, $P(L)_n$ is the vector
space with basis all $L$-labelled $n$-tangles. Any subset
$R \subseteq P(L)$ generates a planar ideal $I(R)$ in $P(L)$
and the corresponding quotient planar algebra is denoted $P(L,R)$.

Let $P$ be a subfactor planar algebra of depth at most $k$ and $B$ be a
basis of $P_k$.
For $b \geq 0$, let $B^{\times b}$ be the Cartesian product of $b$
copies of $B$ for $b>0$ and to be $\{1\}$ for $b=0$.

Let $L = \coprod_{n \in Col} L_n$ where the only non-empty $L_n$ 
is $L_k = B$.
Consider the universal planar algebra $P(L)$.
The templates of Figure \ref{templates} specify a subset $R \subseteq P(L)$ as follows.
Fix one of the templates, say $S \Rightarrow T$, where
$S$ has $b$ internal boxes and $T$ has $c$ internal boxes.
Note that the colour of the internal boxes (if any) of each of $S$ and $T$ is $k$.
For  $(x_1,\cdots,x_b) \in B^{\times b}$
write 
$$Z_S^P(x_1 \otimes \cdots \otimes x_b) = \sum_{\{(y_1,\cdots,y_c) \in B^{\times c}\}} \lambda^{(y_1,\cdots,y_c)}Z_T^P(y_1 \otimes \cdots \otimes y_c),$$
for (not necessarily unique) $\lambda^{(y_1,\cdots,y_c)} \in {\mathbb C}$ (with the obvious interpretations if $b$ or $c$ is 0).
This can be done since $S \Rightarrow T$ holds for $(P,B)$. Now consider the following element of $P(L)$ :
$$
S(x_1,\cdots,x_b) - \sum_{\{(y_1,\cdots,y_c) \in B^{\times c}\}} \lambda^{(y_1,\cdots,y_c)}T(y_1,\cdots,y_c),$$
where $S(x_1,\cdots,x_b)$ denotes the tangle $S$ with boxes
labelled $x_1,\cdots,x_b$ etc.
Consider the collection consisting of one such element of $P(L)$ for
each $(x_1,\cdots,x_b) \in B^{\times b}$ and take the union of these
collections over all templates $S \Rightarrow T$ of Figure \ref{templates}.
This (clearly finite) subset of $P(L)$ is what we will call $R$. Note that $R$
is not a uniquely determined set but depends on choices.
We will call this a set of relations determined by the templates of Figure \ref{templates}.

\begin{Theorem}\label{main}
Let $P$ be a subfactor planar algebra 
of finite depth at most $k$. Let $B$ be a fixed basis of $P_k$. Consider
the labelling set $L = \coprod_{n \in Col} L_n$
where the only non-empty $L_n$ is $L_k = B$. Let $R$ be any set
of relations in $P(L)$ determined by the templates in Figure \ref{templates}. Then,
the quotient planar algebra 
$P(L,R) \cong P$.
\end{Theorem}

\begin
{proof}
Consider the natural surjective planar algebra morphism from the
universal planar algebra $P(L)$ to $P$ defined uniquely by
taking a labelled $k$-box to itself regarded as an element of $P$.
Equivalently, under this morphism, for any tangle $S$ all of whose
internal boxes are of colour $k$, $S(x_1,\cdots,x_b) \mapsto Z^P_S(x_1 \otimes \cdots \otimes x_b)$.
Since the relations $R$ were chosen to hold in $P$, this morphism
factors through the quotient planar algebra $P(L,R)$ thus yielding a
surjective planar algebra morphism $P(L,R) \rightarrow P$. We wish
to see that this is an isomorphism.

For $n \in Col$, let 
$Q_n$ be the subspace of $P(L,R)_n$ spanned by all
$Z^{P(L,R)}_{T^n}(x_1 \otimes x_2 \otimes \cdots \otimes x_{n-k+1})$ for $x_1,\cdots,x_{n-k+1} \in B$, if $n \geq k$ 
or the subspace spanned by all $Z^{P(L,R)}_{T^n}(x)$ for $x \in B$, if
$n < k$. 
Let ${\mathcal T}$ be the set of all $(n_0;n_1,\cdots,n_b)$
tangles $T$ such that
(i) if $b>0$, then $Z_T(Q_{n_1} \otimes \cdots \otimes Q_{n_b}) \subseteq Q_{n_0}$, and
(ii) if $b=0$, then $Z_T(1) \in Q_{n_0}$.
Chasing definitions shows that ${\mathcal T}$ may be equivalently described as the
set of $(n_0; n_1,\cdots,n_b)$-tangles $T$ for which
$T \circ_{(D_1,\cdots,D_b)}(T^{n_1},\cdots,T^{n_b}) \Rightarrow T^{n_0}$ holds
for $(P(L,R),B)$.
We will show that ${\mathcal T}$ consists of
all tangles, or equivalently, that $Q$ is a planar subalgebra of
$P(L,R)$.

For this, we appeal to the main result of \cite{KdySnd2004} which
states that if ${\mathcal T}$ is a collection of tangles that is closed
under composition (whenever it makes sense) and contains the tangles $1^{0_\pm}$, $E^n$ for $n \geq 2$, $ER_{n+1}^n, M_{n,n}^n,
I_n^{n+1}$ for all $n \in Col$ and $EL^n_n$ for all $n \geq 1$, then
${\mathcal T}$ contains all tangles.

To verify the hypotheses for our ${\mathcal T}$, observe first that
by definition if $T \in {\mathcal T}$ is a $(n_0;n_1,\cdots,n_b)$ tangle
and $S  \in {\mathcal T}$ is any $n_i$-tangle for $i >0$, then, $T\circ_{D_i}S \in {\mathcal T}$. Thus ${\mathcal T}$ is closed under composition.
That the other hypotheses hold for ${\mathcal T}$ follows from the
observation that the templates of Figure \ref{templates} hold for
$(P(L,R),B)$ by construction of $R$ and therefore their consequences (3),(7),(10),(11),(12) of Theorem \ref{consequences} also hold.

It follows that $Q$ is a planar subalgebra of $P(L,R)$. Since it contains all generators of $P(L,R)$, it is the whole of $P(L,R)$.
In particular, $P(L,R)_k$ which maps onto $P_k$ equals $Q_k$
which is spanned by $B$ and so $P(L,R)_k$ maps isomorphically
onto $P_k$. It easily follows that the map $P(L,R)_n \rightarrow P_n$
is an isomorphism for $n \leq k$.

%


For $n \geq k$, observe that Corollary \ref{cor:tensor} applies to
$P(L,R)$ since the depth template holds for $(P(L,R),B)$.
Hence we have an isomorphism of ${P(L,R)_{k-1}}-{P(L,R)_{k-1}}$-bimodules
$$
P(L,R)_k \otimes_{P(L,R)_{k-1}} P(L,R)_k \otimes_{P(L,R)_{k-1}} \cdots \otimes_{P(L,R)_{k-1}} P(L,R)_k \rightarrow P(L,R)_n,
$$
and therefore an isomorphism of ${P_{k-1}}-{P_{k-1}}$-bimodules
$$
P_k \otimes_{P_{k-1}} P_k \otimes_{P_{k-1}} \cdots \otimes_{P_{k-1}} P_k \rightarrow P(L,R)_n.
$$
Since the left side is, by Corollary  \ref{cor:tensor} applied to $P$,
isomorphic to $P_n$ while the right side maps onto $P_n$, it follows
that $P(L,R)_n$ maps isomorphically to $P_n$ also for
all $n \geq k$.
\end{proof}

\section{On single generation}

Rather surprisingly, the fact that finite depth subfactor planar
algebras are singly generated has a simple proof.
%

\begin{proposition}\label{single}
Let $P$ be a subfactor planar algebra of finite depth at most $k$. Then $P$ is
generated by a single $2k$-box.
\end{proposition}

\begin{proof} As a planar algebra, $P$ is generated by $P_k$. Since
$P_k$ is a finite-dimensional $C^*$-algebra, it is singly generated, by say $x \in P_k$. By adding a multiple of $1_k$ to $x$, we may
assume without loss of generality that $\tau(x) \neq 0$ (recall that
$\tau(\cdot)$ is the normalised picture trace on $P$).
Thus the planar algebra generated by $x$ and $x^*$ contains $P_k$
and must be the whole of $P$.
Now consider the element $z \in P_{2k}$ defined by
\begin{figure}[!h]
\begin{center}
\psfrag{k}{\large $k$}
\psfrag{x}{\huge $x$}
\psfrag{x*}{\huge $x^*$}
\resizebox{3cm}{!}{\includegraphics{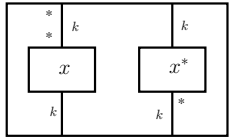}}
\end{center}
\caption{Definition of $z \in P_{2k}$}
\label{generator}
\end{figure}
Figure \ref{generator}. It should be clear that applying suitable annular tangles to $z$ yields non-zero (since $\tau(x) \neq 0$) multiples
of $x$ and $x^*$.  Hence the planar
subalgebra of $P$ generated by $z$ contains both $x$ and $x^*$
and consequently is $P$.
\end{proof}

\begin{remark} Let $d$ be a fixed positive integer. For $n \in Col$, letting $P(d)_n$ be the vector space
spanned by all $n$-tangles whose only internal boxes are of colour $d$, there is an obvious planar algebra structure on $P(d)$. 
What Proposition \ref{single} asserts is that $P(2k)$ maps onto
any subfactor planar algebra of depth $k$.
\end{remark}

It is natural to ask whether when a finite depth planar algebra $P$
is presented as a quotient of $P(2k)$ as above, the kernel is a finitely
generated planar ideal. A standard proof shows that this is indeed so.


\begin{proposition}\label{equiv}
Let $P$ be a planar algebra and suppose that for finite label sets
$L$ and $\tilde{L}$ there are surjective planar algebra maps
$\pi : P(L) \rightarrow P$ and $\tilde{\pi} : P(\tilde{L}) \rightarrow P$.
The ideal $I = ker(\pi)$ is a finitely generated planar ideal of $P(L)$
if and only if $\tilde{I} = ker(\tilde{\pi})$ is a finitely generated planar
ideal of $P(\tilde{L})$.
\end{proposition}

\begin{proof} First note that universality of $P(L)$ and $P(\tilde{L})$
yield (possibly non-unique) planar algebra maps $\phi : P(L) \rightarrow P(\tilde{L})$ and
$\tilde{\phi} : P(\tilde{L}) \rightarrow P(L)$ that satisfy $\tilde{\pi} \circ \phi = \pi$ and $\pi \circ \tilde{\phi} = \tilde{\pi}$.

By symmetry, it suffices to prove one implication. Suppose that $\tilde{I} = I(\tilde{R})$ for a finite subset $\tilde{R} \subseteq P(\tilde{L})$. Let $R= \tilde{\phi}(\tilde{R}) \cup \{ {x} - \tilde{\phi} {\phi}({x}) : {x} \in {L}\}$, which
is clearly a finite subset of $P({L})$.
We claim that ${I} = I({R})$. 

Clearly ${R} \subseteq {I}$ and so $I({R}) \subseteq {I}$. The other inclusion needs a little work. 
First observe that $\{ {x} - \tilde{\phi} {\phi}({x}) : {x} \in {L}\} \subseteq {R}$ implies that for all ${z} \in P(L)$,
${z} - \tilde{\phi} {\phi}({z}) \in I({R})$. To see this we may reduce easily to the case that ${z} = T({x}_1,\cdots,{x}_b)$ where $T$ is a $(n_0;n_1,\cdots,n_b)$-tangle and ${x}_i \in {L_{n_i}}$. Then
$${z} - \tilde{\phi} {\phi}({z}) = Z_T^{P({L})}({x}_1 \otimes \cdots \otimes {x}_b) - Z_T^{P({L})}(\tilde{\phi} {\phi}({x}_1) \otimes \cdots \otimes \tilde{\phi} {\phi}({x}_b)).$$
This may be expressed as a telescoping sum of $b$ terms indexed
by $k=1,2,\cdots,b$ where the $k^{th}$ term is given by
$$Z_T^{P({L})}(\tilde{\phi} {\phi}({x}_1) \otimes \cdots \otimes \tilde{\phi} {\phi}({x}_{k-1}) \otimes ({x}_k-  \tilde{\phi} {\phi}({x}_{k})) \otimes {x}_{k+1} \otimes \cdots \otimes {x}_b)$$
Each of these terms is clearly in the planar ideal generated by 
$\{ {x} - \tilde{\phi} {\phi}({x}) : {x} \in {L}\}$ and hence in $I({R})$. Therefore ${z} - \tilde{\phi} {\phi}({z}) \in I({R})$.

Say ${z} \in {I}$, so that ${\pi}({z}) = 0$. Then ${\phi}({z}) \in ker(\tilde{\pi}) = \tilde{I} = I(\tilde{R})$, i.e., ${\phi}({z})$ is in the planar
ideal generated by $\tilde{R}$. It follows that $\tilde{\phi} {\phi}({z})$ is in
the planar ideal generated by $\tilde{\phi}(\tilde{R})$ and therefore in $I({R})$. 
Since ${z} - \tilde{\phi} {\phi}({z}) \in I({R})$, 
we also have ${z} \in I({R})$ and 
the proof is finished.
\end{proof}

A direct consequence of Theorem \ref{main} and Propositions \ref{single} and \ref{equiv} is the following corollary.

\begin{Corollary}
If $P$ is a subfactor planar algebra of finite depth at most $k$,
then $P$ is generated by a single $2k$-box subject to finitely many relations.\qed
\end{Corollary}

\section*{Acknowledgments}
We thank V. S. Sunder for useful discussions.

\end{document}